\documentclass[reqno,12pt]{amsart}

\usepackage{color} 
\usepackage{ifpdf}
\ifpdf
    \usepackage[pdftex]{graphicx}
    \usepackage[pdftex]{hyperref}
    \hypersetup{
        unicode=false,          
        pdftoolbar=true,        
        pdfmenubar=true,        
        pdffitwindow=false,     
        pdfstartview={FitH},    
        pdftitle={SP Article},      
        pdfauthor={Sara Pollock},   
        pdfsubject={Mathematics},    
        pdfcreator={Sara Pollock},  
        pdfproducer={Sara Pollock}, 
        pdfkeywords={PDE, analysis, numerical analysis}, 
        pdfnewwindow=true,      
        colorlinks=true,        
        linkcolor=red,          
        citecolor=blue,         
        filecolor=magenta,      
        urlcolor=cyan           
    }

\else
    \usepackage{graphicx}
    \usepackage{pstricks}
    
    \newcommand{\href}[2]{#2}
\fi

\usepackage{times}
\usepackage{amsfonts}
\usepackage[font={small}]{caption}

\usepackage{amsmath}
\usepackage{amsthm}
\usepackage{amssymb}
\usepackage{amsbsy}
\usepackage{amscd}
\usepackage{extarrows}

\usepackage{enumerate}
\usepackage{verbatim}
\usepackage{subfigure}

\newtheorem{theorem}{Theorem}[section]
\newtheorem{corollary}[theorem]{Corollary}
\newtheorem{lemma}[theorem]{Lemma}

\newtheorem{assumption}[theorem]{Assumption}

\newtheorem{remark}[theorem]{Remark}

\numberwithin{equation}{section}  

  \newcounter{mnote}
  \let\oldmarginpar\marginpar
    \renewcommand\marginpar[1]{\-\oldmarginpar[\raggedleft\footnotesize #1]%
    {\raggedright\footnotesize #1}}

\definecolor{myblue}{rgb}{0.2,0.2,0.7}
\definecolor{mygreen}{rgb}{0,0.6,0}
\definecolor{mycyan}{rgb}{0,0.6,0.6}
\definecolor{myred}{rgb}{0.9,0.2,0.2}
\definecolor{mymagenta}{rgb}{0.9,0.2,0.9}
\definecolor{mywhite}{rgb}{1.0,1.0,1.0}
\definecolor{myblack}{rgb}{0.0,0.0,0.0}

\renewcommand{\div}{{\operatorname{div}}}
\renewcommand{\d}{{\operatorname{d}}}

\newcommand{\R}{{\mathbb R}}       

\newcommand{\cD}{{\mathcal D}}

\newcommand{\cI}{{\mathcal I}}

\newcommand{\cN}{{\mathcal N}}

\newcommand{\cT}{{\mathcal T}}

\newcommand{\cV}{{\mathcal V}}

\DeclareMathAlphabet{\mathpzc}{OT1}{pzc}{m}{it}

\newcommand{\f}{\frac}

\newcommand{\of}{\text{ of }}

\newcommand{\on}{\text{ on }}

\newcommand{\an}{\text{ and }}

\newcommand{\inn}{\text{ in }}

\newcommand{\tforall}{\text{ for all }}

\newcommand{\ale}{\text{a.e.\,}} 

\newcommand{\rest}{\big|}

\newcommand{\grad}{\nabla} 

\newcommand{\goto}{\rightarrow}

\newcommand{\pa}{\partial}

\definecolor{blue}{rgb}{0.2,0.2,0.7}
\definecolor{red}{rgb}{0.7,0.3,0.1}
\definecolor{cyan}{rgb}{0.2,0.5,0.6}
\usepackage{mathtools}
\usepackage{stmaryrd}

\begin{document}

\title[Uniqueness without globally small meshsize]
      {Uniqueness of discrete solutions of nonmonotone PDEs without a globally fine mesh condition
}

\author[S. Pollock]{Sara Pollock}
\email{sara.pollock@wright.edu}

\author[Y. Zhu]{Yunrong Zhu}
\email{zhuyunr@isu.edu}

\address{Department of Mathematics and Statistics\\
         Wright State University\\ 
         Dayton, OH 45435}

\address{Department of Mathematics and Statistics\\
         Idaho State University\\ 
         Pocatello, ID 83209}
\thanks{YZ was supported in part by NSF DMS 1319110.}

\date{\today}



\begin{abstract}
Uniqueness of the finite element solution for nonmonotone quasilinear problems of
elliptic type is established in one and two dimensions.
In each case, we prove a comparison theorem based on locally bounding the 
variation of the discrete solution over each element.  The uniqueness 
follows from this result, and does not require a globally small meshsize.  
\end{abstract}

\maketitle




\section{Introduction}\label{sec:intro} 
We consider the piecewise linear finite element (FE)
approximation of the quasilinear elliptic partial 
differential equation (PDE)
\begin{align}\label{eqn:PDE_strong}
-\div (\kappa(x,u) \grad u) = f ~\inn~ \Omega.
\end{align}
The following assumptions on 
the diffusion coefficient $\kappa(x,s)$, for $x \in \Omega$ and $s \in \R$
are made throughout the remainder of the paper. 
Further assumptions on the source function $f$ are stated as needed in the 
following sections, along with the boundary conditions corresponding to 
\eqref{eqn:PDE_strong}.
\begin{assumption}\label{A1:elliptic}
Assume $\kappa(x,s)$ is a Carath\'eodory function, and assume 
there are constants $0 < k_\alpha < k_\beta$ with
\begin{align}\label{eqn:A1:elliptic}
k_\alpha < \kappa(x, s) < k_\beta,
\end{align} 
for all $s \in \R$, and $\ale x \in \Omega$.
\end{assumption}
Further assume the diffusion coefficient $\kappa$ satisfies the following Lipschitz 
condition.
\begin{assumption}\label{A1:Lipschitz}
There is a constant $L_0 > 0$ with
\begin{align}\label{eqn:A1:Lipschitz}
|\kappa(x,s) - \kappa(x,t)| \le L_0 |s - t|, ~\tforall s, t \in \R, 
                             ~\ale x \in \Omega.
\end{align}
\end{assumption}
Existence and uniqueness of smooth solutions, $u \in C^2(\bar \Omega)$, 
to the continuous problem \eqref{eqn:PDE_strong} were established for the Dirichlet 
problem in \cite{DoDuSe71}. Those results were extended in \cite{HlKrMa94} to 
cover existence and uniqueness of weak solutions $u \in H^1(\Omega)$ under
mixed Dirichlet/Robin conditions on the boundary, and for Dirichlet problems
in \cite{AnCh96b}. 
Uniqueness results and comparison principles for more general classes of 
quasilinear problems are also found in Chapter 10 of \cite{GilbargTrudinger}, 
and in \cite{AlBeMe10}, and the references therein. 
As demonstrated by the counterexamples shown in \cite{AnCh96a}, 
uniqueness for the discrete problem may fail, 
even where it holds for the corresponding continuous problem.

Let  $\Omega \subset \R^n$, $n \ge 1$, be an open bounded polygonal domain, 
(an interval, if $n = 1$) with boundary $\pa \Omega$. 
Let $\cV$ denote the discrete space of continuous piecewise linear
functions subordinate to 
a conforming simplicial partition $\cT$ of $\Omega$ such that functions
$v\in\cV$ vanish on the prescribed Dirichlet part of the boundary. 
We consider the approximation of solutions
to the weak form of \eqref{eqn:PDE_strong} by functions $u \in \cV$.
We refer to Theorem 1.2 of \cite{AnCh96a} and Theorem 2.6 of \cite{HlKrMa94} 
to establish the existence of a solution to the discrete 
Dirichlet and respectively mixed Dirichlet/Neumann problems.
Succinctly, let $f \in H^{-1}(\Omega)$ in the Dirichlet
case, or $f \in L_2(\Omega)$ for the case of mixed boundary conditions, with
bounded, measurable data on the boundary.
Then under Assumptions \ref{A1:elliptic} and \ref{A1:Lipschitz}, 
there exists a solution $u \in \cV$ to \eqref{eqn:PDE_strong} satisfying the 
prescribed conditions on the boundary. The upper bound $k_\beta$ in Assumption ~\ref{A1:elliptic} is used in establishing existence, and not elsewhere
in the remainder of this paper.

We now turn our attention to sufficient conditions to establish uniqueness
of the discrete solution.
In the recent investigation \cite{AbVi12}, regarding the effects of numerical
quadrature on nonmonotone nonlinear problems over $\Omega \subset \R^d,$ $d \le 3$, 
the condition of a globally sufficiently fine mesh as required
in \cite{AnCh96a} is referred to as essential for establishing the
uniqueness of the discrete solution.  
Within the context of adaptive finite element methods, convergence of
a Discontinuous Galerkin method for this class of problems is shown in 
\cite{BiGi13}, under the assumption that the mesh is globally sufficiently fine;
more precisely, under the assumptions of \cite{DoDu75} which develops local 
convergence theory for the piecewise linear finite element solution.
However, recent work by one of the present authors
on the numerical solution of the same class of 
nonmonotone problems by means of adaptive regularized methods 
\cite{Pollock15a,Pollock15b,Pollock16a}, demonstrates that a solution to the 
discrete problem may be recovered under only local adaptive mesh refinement.  
This raises the
question of whether the solution to the discrete problem can be shown unique
under only local criteria.  This paper answers the question in the positive.  

Inspired by the analysis of \cite{AnCh96a}, the discrete finite element
solution is shown here to be unique so long as the variation of the discrete solution
over each mesh element satisfies a given bound related to the constants in 
Assumptions \ref{A1:elliptic} and \ref{A1:Lipschitz}.  
To our knowledge, this is the first result indicating uniqueness
of the discrete solution for this class of problems that does not require a 
globally fine mesh. The results here are viewed as an important step towards
establishing the convergence of an efficiently computed discrete solution 
by means of adaptive mesh refinement.  The issue of convergence of the 
discrete solution to the PDE solution under adaptive refinement is not addressed
here, as the analysis of appropriate {\em a posteriori} error indicators is 
nontrivial and merits separate investigation. 
 
Our approach to establishing uniqueness under local conditions 
is based upon first establishing a comparison theorem from which
the uniqueness of the solution immediately follows.  
We remark that discrete maximum principles have been established for this 
problem and some of its generalizations without requirements on the global
meshsize, as in \cite{KaKo05,Karatson.J;Korotov.S;Krizek.M2007,WaZh12} and the references
therein. These are important tools for studying nonlinear problems because 
they provide {\em a priori} bounds on the discrete solution. 
However, maximum principles are not sufficient to establish uniqueness for the nonlinear PDE, 
even in the continuous setting.

The remainder of the paper is organized as follows. In Section \ref{sec:1D} we first
establish a comparison theorem followed by a uniqueness result for the one 
dimensional Dirichlet problem.  The analysis is then extended to the case of
mixed boundary conditions. In Section \ref{subsec:1Dcounter} we demonstrate
the local criterion for uniqueness on a counterexample where two discrete solutions
are known to exist.  Next, in Section \ref{sec:2D} we first review some useful
facts about the finite element basis functions, then in \ref{subsec:2Dunique} 
we prove a comparison theorem followed by a uniqueness result for the mixed
Dirichlet/Neumann problem. Finally, we demonstrate the local uniqueness condition
on a known two dimensional counterexample with two discrete solutions.
\section{Uniqueness in the one dimensional case}
\label{sec:1D}
The one dimensional quasilinear equation is given by the ordinary differential equation
\begin{align}\label{eqn:strongPDE_1D}
-\f{\d}{\d x}\left( \kappa(x,u) \f{\d}{\d x} u \right)= f, \inn \Omega \subset \R,
\end{align}
subject to appropriate boundary conditions as described in the main theorems below.

As in \cite{AnCh96a} let $\Omega = (a,b)$, with a subdivision
\begin{align}\label{eqn:subdivision1D}
a = a_0 < a_1 < \ldots < a_{n-1} < a_n = b,
\end{align}
where the mesh spacing is not assumed to be uniform.
Define the discrete space $\cV_0 \subset H^1_0(\Omega)$ as the piecewise linear 
Lagrange finite element space
subordinate to subdivision \eqref{eqn:subdivision1D}, and 
subject to homogeneous Dirichlet boundary conditions: $v(a) = v(b) = 0$ for all
$v \in \cV_0$. 
Let $(u,v)$ denote the integral $\int_\Omega u v$, and let $v' = \d v/\d x$. 

\begin{theorem}[Comparison Theorem in 1D]\label{thm:1dcompare}
Assume $f_1, f_2 \in H^{-1}(\Omega)$ satisfy 
\begin{align}\label{eqn:1dun001}
f_1 \le f_2,
\end{align}
in the $H^{-1}(\Omega)$ sense, that is $\int_\Omega f_1 v \le \int_\Omega f_2 v$ for all
non-negative functions $v \in H_0^1(\Omega)$.
Let $u_i \in \cV_0, ~ i = 1,2,$ denote the solution to the respective weak forms of 
\eqref{eqn:strongPDE_1D}, with homogeneous Dirichlet boundary conditions. 
In particular
\begin{align}\label{eqn:1dun002}
(\kappa(x,u_i) u_i', v') = (f_i,v), \tforall v \in \cV_0, \quad i = 1,2.
\end{align}
Define the intervals $\cI_k = (a_{k-1}, a_k), ~k = 1, \ldots, n$, and 
let $h_k = a_k - a_{k-1}$, the length of each respective interval
$\cI_k, ~ k = 1, \ldots, n$. 
Then under the condition
\begin{align}\label{eqn:1dun003}
\max_{0 \le k \le n-1} |u_2(a_k) - u_2(a_{k+1})| < \f{2 k_\alpha}{L_0},
\end{align} 
it holds that 
\begin{align}\label{eqn:1dun004}
u_1 \le u_2.
\end{align}
\end{theorem}
The proof follows similarly to Theorem 2.1 in 
\cite{AnCh96a} with an updated choice of test-function and the avoidance
of global norm estimates. 
\begin{proof}
By \eqref{eqn:1dun002}, subtracting the cases $i=1$ and $i=2$,
\begin{align}\label{eqn:1dun005}
(\kappa(x,u_1) (u_1 - u_2)', v') 
  = \left((\kappa(x,u_2) - \kappa(x,u_1)) u_2', v' \right)
    +(f_1 - f_2, v), \tforall v \in \cV_0.
\end{align}
We prove the theorem by contradiction. Suppose that $w(x)= u_{1}(x) - u_{2}(x) > 0$ anywhere on $a \le x \le b$, then we can define the test function $v \in \cV_0$ as follows.
Let $i$ denote the smallest index such that $w(a_i) > 0$; and, 
let $j$ be the smallest index with $j \ge i$, and such that $w(a_{j+1}) \le 0$. 
Then define the piecewise linear function $v$ by its nodal values
\begin{align}\label{eqn:1dun006}
\left\{ \begin{array}{ll}
v(a_k) = 1, & k = i, \ldots, j, \\
v(a_k) = 0, & \text{otherwise}.
\end{array}\right.
\end{align}
As $u_1$ and $u_2$ satisfy the same boundary conditions, 
we have $v(a_0) = 0$ and $v(a_n)=0$ so that $v$ is a valid test-function, 
that is $v \in \cV_0$. 
As $v' = 0$ on $\Omega \setminus \{ \cI_i \cup \cI_{j+1}\}$,
the left hand side of \eqref{eqn:1dun005} satisfies
\begin{align}\label{eqn:1dun007}
(\kappa(x,u_1) (u_1 - u_2)',v') 
  & = \int_{\cI_i} \kappa(x,u_1)  w' v' 
    + \int_{\cI_{j+1}} \kappa(x,u_2) w'  v'
    \nonumber \\
  & = \int_{\cI_i} \!\kappa(x,u_1) \f{ w(a_i) - w(a_{i-1})  }{h_i^2} 
    + \int_{\cI_{j+1}} \!\!\!\!\kappa(x,u_1) \f{ w(a_j) - w(a_{j+1})  }{h_{j+1}^2} 
    \nonumber \\
  & \ge k_\alpha \left\{ \f{ w(a_i) - w(a_{i-1})  }{h_i}
                       + \f{ w(a_j) - w(a_{j+1})  }{h_{j+1}} \right\}, 
\end{align}
where the ellipticity condition ~\eqref{eqn:A1:elliptic} was applied, as was 
the fact $w(a_i) > w(a_{i-1})$, and $w(a_j) > w(a_{j+1})$, from the definition of
$v$, alongside $v' = 1/h_i$ on $\cI_i$, and $v' = -1/h_{j+1}$ on $\cI_{j+1}$.

Similarly, the first term on the right hand side of ~\eqref{eqn:1dun005} satisfies
\begin{align}\label{eqn:1dun008}
 &\left( (\kappa(x,u_2) - \kappa(x,u_1)) u_2', v' \right) 
  \nonumber \\ 
  & = \int_{\cI_i }(\kappa(x,u_2) - \kappa(x,u_1)) u_2' v' 
      +\int_{\cI_{j+1} }(\kappa(x,u_2) - \kappa(x,u_1)) u_2'  v' 
  \nonumber \\ 
  & = \int_{\cI_i} ( \kappa(x,u_2) - \kappa(x,u_1) ) \f{u_2(a_i) - u_2(a_{i-1})}{h_i^2}
  \nonumber \\
  & -\int_{\cI_{j+1}} ( \kappa(x,u_2) - \kappa(x,u_1) ) 
    \f{u_2(a_{j+1}) - u_2(a_{j})}{h_{j+1}^2}
  \nonumber \\
  & \le \f{L_0}{h_i^2}|u_2(a_{i}) - u_2(a_{i-1}) | \left( \int_{\cI_{i}} |w|\right)
  +  \f{L_0}{h_{j+1}^2}|u_2(a_{j+1}) - u_2(a_{j})|\left( \int_{\cI_{j+1}} |w| \right),
\end{align}
where both the Lipschitz property \eqref{eqn:A1:Lipschitz} of $\kappa$, and the
value of $v$ on $\cI_i$ and $\cI_{j+1}$, as in \eqref{eqn:1dun007}, were applied. 
Noticing that $w(a_{i-1})\le 0$ and $w(a_i) > 0$, and similarly of $w(a_{j+1})$ and 
$w(a_j)$, it follows from the piecewise linearity of $w$ that
\begin{align}\label{eqn:1dun008a}
\int_{\cI_i}|w| \le \f{h_i}{2} (w(a_{i}) - w(a_{i-1})), \an 
\int_{\cI_{j+1}}|w| \le \f{h_{j+1}}{2} (w(a_{j}) - w(a_{j+1})).
\end{align}
Applying \eqref{eqn:1dun008a} to \eqref{eqn:1dun008}, yields 
\begin{align}\label{eqn:1dun008b} 
\left( (\kappa(x,u_2) - \kappa(x,u_1)) u_2', v' \right) 
  & \le \f{L_0}{2 h_i}(w(a_i) - w(a_{i-1}))|u_2(a_i) - u_2(a_{i-1})|
  \nonumber \\
  &  +  \f{L_0}{2 h_{j+1}}(w(a_j) - w(a_{j+1}))|u_2(a_j) - u_2(a_{j+1})|.
\end{align}
Putting together \eqref{eqn:1dun007} and \eqref{eqn:1dun008b} with \eqref{eqn:1dun005}
yields
\begin{align}\label{eqn:1dun009}
k_\alpha  \left\{ \f{ w(a_i)\!-\!w(a_{i-1})  }{h_i} 
        \! +\! \f{ w(a_j)\!-\!w(a_{j+1})  }{h_{j+1}} \right\} \!
  & \le \f{L_0}{2 h_i}(w(a_i) - w(a_{i-1}))|u_2(a_i) - u_2(a_{i-1})|
  \nonumber \\
   &\!+\!  \f{L_0}{2 h_{j+1}}(w(a_j) \!-\! w(a_{j+1}))|u_2(a_j) \!-\! u_2(a_{j+1})| 
  \nonumber \\
  & + (f_1 - f_2,v).
\end{align}
Rearranging terms yields
\begin{align}\label{eqn:1dun010}
\f {1}{h_i} &\left\{ k_\alpha - \f{L_0}{2} | u_2(a_i) - u_2(a_{i-1}) |\right\} 
                                    ( w(a_i) - w(a_{i-1})     )
\nonumber \\ 
&+\f {1}{h_{j+1}} \left\{ k_\alpha - \f{L_0}{2} | u_2(a_{j+1}) - u_2(a_{j}) |\right\} 
                                    ( w(a_j) - w(a_{j+1})     )
\nonumber \\
& \le (f_1 - f_2, v).
\end{align}
Noting again the terms $(w(a_i) - w(a_{i-1}))$ and $(w(a_{j}) - w(a_{j+1}))$, are both
strictly positive, and $(f_1 - f_2,v) \le 0$, this yields a contradiction under the 
condition \eqref{eqn:1dun003}, from which it is clear the test-function $v$
cannot be so defined, and the conclusion follows.
\end{proof}

The main purpose of Theorem~\eqref{thm:1dcompare} is to establish the following 
immediate corollary, namely the uniqueness of the one dimensional linear
FE solution without a global meshize condition.  The sufficient condition 
used here bounds the change in the FE solution over a single mesh element.  
This of course can be controlled by {\em locally} reducing the meshsize.
As such, the result is suitable for an adaptive mesh refinement algorithm. 

\begin{corollary}[Uniqueness of the FE Solution in 1D]
\label{cor:1d_unique}
Let $f \in H^{-1}(\Omega)$, 
and let $u \in \cV_0$ denote a solution to the weak form of 
\eqref{eqn:strongPDE_1D}, with homogeneous Dirichlet boundary conditions. 
In particular
\begin{align}\label{eqn:1dun011}
(\kappa(x,u) u', v') = (f,v), \tforall v \in \cV_0.
\end{align}
Then under the condition 
\begin{align}\label{eqn:1dun006a}
\max_{0 \le k \le n-1} |u(a_k) - u(a_{k+1})| < \f{2 k_\alpha}{L_0},
\end{align}
({\em cf.,} \eqref{eqn:1dun003}), the linear FE solution $u$
is the unique solution to \eqref{eqn:1dun011}.
\end{corollary}
\begin{proof}
Let $f_1 = f_2 = f$ in Theorem ~\ref{thm:1dcompare}. Then, let $u_1$ and $u_2$ be
two respective solutions to the problem \eqref{eqn:1dun011}. Applying the results
of Theorem \ref{thm:1dcompare} yields both $u_1 \le u_2$ and $u_2 \le u_1$.  
\end{proof}
Notably, the argument does not
depend on the global meshsize, rather the variation in the discrete 
piecewise linear solution over each element. As per the following remark,
this can be rephrased as a condition involving the local meshsize.

\begin{remark} The condition \eqref{eqn:1dun006a} is local, and can be checked
computationally. 
These two qualities make it amenable for use in an  adaptive 
algorithm to assure uniqueness of the discrete solution.
The condition can equivalently be written as
\begin{align*}
|u_2'|_{\cI_k} < \f{2 k_\alpha}{h_k L_0}, \tforall k = 1, \ldots n.
\end{align*}
In other words, the mesh should be {\em locally} fine, where the slope 
of the solution is steep.
\end{remark}

The next corollary generalizes Theorem~\ref{thm:1dcompare} 
to the case of mixed Dirichlet-Neumann boundary conditions.
\begin{corollary}[Comparison Theorem for Mixed Boundary Conditions in 1D]
\label{cor:compare1d_mixed}
Assume $f_1, f_2 \in L_2(\Omega)$ satisfy $f_1 \le f_2$ in $\Omega$. 
Define $\cV_{0,b}$ as the piecewise linear Lagrange finite element space subordinate
to subdivision ~\eqref{eqn:subdivision1D}, with $v(b) = 0$ for all $v \in \cV_{0,b}$.
Let $u_i \in \cV_{0,b}, ~ i = 1,2$ denote the solution to the respective weak forms 
of \eqref{eqn:strongPDE_1D} with the mixed boundary conditions
\begin{align}\label{eqn:mixedBCa_1D}
u(b) = 0, ~\an \kappa(a,u(a))u'(a) = \psi_a,
\end{align}
for a given $\psi_a \in \R$.
In particular
\begin{align}\label{eqn:weak_mixedBCa_1D}
(\kappa(x,u_i) u_i',v') = (f_i,v) + \psi_a v(a), \tforall v \in \cV_{0,b}, 
\quad i = 1,2.
\end{align}
Define the intervals $\cI_k$, and their respective lengths $h_k$ as in
Theorem \ref{thm:1dcompare}, for $k = 1, \ldots, n$.  Then under the condition
\eqref{eqn:1dun003}, it holds that $u_1 \le u_2$.
\end{corollary}
The proof follows again by construction of an appropriate test function $v$.
\begin{proof}
By \eqref{eqn:weak_mixedBCa_1D}, and subtracting the cases $i = 1$ and $i=2$,
\begin{align}\label{eqn:1dmun001}
(\kappa(x,u_1) (u_1 - u_2)', v')   
  = \left((\kappa(x,u_2) - \kappa(x,u_1)) u_2', v' \right)
  +(f_1 - f_2, v),  \tforall v \in \cV_{0,b}.
\end{align}
As before, we define $w = u_1 - u_2$ and prove the result by contradiction. Supposing that $w(x) > 0$ anywhere on $a \le x \le b$, 
two cases are now possible. In the first case, 
there is a least index $i < n$ with $w(a_{i-1}) \le 0$ and $w(a_{i}) > 0$. Then, because
$w(b) = 0$, there is certain to be a least index $j \ge i$ for which $w(a_j)>0$ and 
$w(a_{j+1} ) \le 0$. The proof then follows that of Theorem \ref{thm:1dcompare},
with $v$ chosen by \eqref{eqn:1dun006}. 

The second case is characterized by $w(a_k) > 0$ for all $k \le j$, for some $j \le n-1$.
In particular, there is some smallest index $j < n$ such that $w(a_j)>0$, and
$w(a_{k}) \le 0$ for all $k = j+1, \ldots, n$. 
It remains then to establish the result assuming this is the case.
Define the test function $v \in \cV_{0,b}$ by 
\begin{align}\label{eqn:1dmun002}
\left\{ \begin{array}{ll}
v(a_k) = 1, & k = 0, \ldots, j, \\
v(a_k) = 0, & k = j+1, \ldots, n.
\end{array}\right.
\end{align} 
Noting that $v' = -1/h_{j+1}$ on $\cI_{j+1}$, and $v' = 0$, otherwise, the argument 
follows as before with the following estimates.
Inequality \eqref{eqn:1dun007} reduces to
\begin{align}\label{eqn:1dmun003}
(\kappa(x,u_1)w',v') 
   = \int_{\cI_{j+1}} \kappa(x,u_2)\f{w(a_j) - w(a_{j+1})}{h_{j+1}^2}
   \ge k_\alpha \left\{ \f{w(a_j) - w(a_{j+1})}{h_{j+1}} \right\}.
\end{align}
Inequality \eqref{eqn:1dun008} reduces to
\begin{align}\label{eqn:1dmun004}
((\kappa(x,u_2) - \kappa(x,u_1))u_2',v') 
  & \le \f{L_0}{2 h_{j+1}} (w(a_j) - w(a_{j+1}))|u_2(a_j) - u_2(a_{j+1})|.
\end{align}
Putting together \eqref{eqn:1dmun003} and \eqref{eqn:1dmun004} with \eqref{eqn:1dmun002},
and rearranging terms, yields
\begin{align}\label{eqn:1dmun005}
\f{1}{h_{j+1}} \left\{k_\alpha - \f{L_0}{2}|u_2(a_{j+1}) - u_2(a_j)| \right\}
  (w(a_j) - w(a_{j+1})) \le (f_1 - f_2,v).
\end{align}
Again noting the strict positivity of $(w(a_{j}) - w(a_{j+1}))$, a contradiction 
is encountered under the satisfaction of the condition \eqref{eqn:1dun003}, 
from which the conclusion follows.
\end{proof}
As in the case of homogeneous Dirichlet boundary conditions, the main 
purpose of Corollary \ref{cor:compare1d_mixed} is to establish the uniqueness of the 
piecewise linear FE solution under conditions suitable for adaptive mesh refinement.
The next corollary for the case of mixed boundary conditions establishes the analogous 
result to Corollary \ref{cor:1d_unique}.
\begin{corollary}[Uniqueness of the FE Solution under Mixed BC in 1D]
\label{cor:unique1d_mixed}
Let $f \in L_2(\Omega)$, and let $u \in \cV_{0,b}$ denote a solution to the weak form
of \eqref{eqn:strongPDE_1D}, with the mixed boundary conditions \eqref{eqn:mixedBCa_1D}.
In particular
\begin{align}\label{eqn:1dmun006}
(\kappa(x,u)u',v') = (f,v) + \psi_a v(a), \tforall v \in \cV_{0,b}.
\end{align}
Then, under the condition \eqref{eqn:1dun006a}, 
the linear FE solution $u$ is the 
unique solution to \eqref{eqn:1dmun006}.
\end{corollary}
\begin{proof}
Let $f_1 = f_2 = f$ in Corollary \ref{cor:compare1d_mixed}. Then let $u_1$ and $u_2$
be two respective solutions to \eqref{eqn:1dmun006}. Applying the results of Corollary
~\eqref{cor:compare1d_mixed} yields the result.
\end{proof}
\begin{remark}
Corollaries \ref{cor:compare1d_mixed} and \ref{cor:unique1d_mixed} trivially generalize
to the mixed boundary value problem with a homogeneous Dirichlet condition at
$x = a$ and a Neumann condition at $x=b$. 
\end{remark}
\subsection{Addressing the counterexample of Andr\'e and Chipot}
\label{subsec:1Dcounter}
The analysis of \cite{AnCh96a} not only established the uniqueness of the 
finite element solution in one and two dimensions under the sufficient condition
of a globally small meshsize, but also provided a counterexample. In this example,
a diffusion coefficient $\kappa(x,u)$ for which $-(\kappa(x,u)u')' = f$ has two 
discrete solutions, $u$ and $2u$, is constructed on a set partition.  
Here, we show the counterexample
in question violates the condition \eqref{eqn:1dun006a}. 
We refer the reader to Section 2 of \cite{AnCh96a} for the complete details of the
example, and explain here how the local condition \eqref{eqn:1dun006a} detects
the possibility of nonuniqueness. 

The diffusion coefficient $\kappa(x,z)$ is constructed to satisfy 
$1/3 \le \kappa(x,z) \le 1$, for all $z \in \R$ and $\ale x \in [0,1]$.  That is, 
\eqref{eqn:A1:elliptic} is satisfied with $k_\alpha = 1/3$.
The coefficient $\kappa(x,z)$ is defined in terms of a function 
$\phi(t) \in C^\infty(0,1)$ which satisfies
\begin{align}\label{eqn:1Dcounter001}
\phi(0) = \phi(1) = 1, \quad \phi'(0) = \phi'(1) = 0, 
\quad 1/3 \le \phi(t) \le 1, \an
\int_0^1 \phi(t) = 1/2.
\end{align}
A first solution $u$ defined on a uniform partition corresponding to 
the notation of \eqref{eqn:subdivision1D}, is assumed to satisfy $u(0) = u(1) = 0$, and $u(a_i) = u_i > 0$, $ i = 1, \ldots, n-1$. 
As demonstrated in equation (2.38) in \cite{AnCh96a}, on the first interval, the
derivative of $\kappa(x,z)$ with respect to its second argument satisfies 
\begin{align}\label{eqn:1Dcounter002}
\left|\f{\pa }{\pa z} \kappa(x,z)  \right| = \left|\f{\phi(t)-\phi(0)}{t u_1} \right|,
\quad a_0 \le x \le a_1, \quad u(x) \le z(x) \le 2u(x).
\end{align}
However, the properties of $\phi$ given by \eqref{eqn:1Dcounter001} require that
the magnitude of the slope of the secant line,
$|\phi(0) - \phi(t)|/t > 2/3$ for some $t \in(a_0,a_1]$, for $z(x)$ in the given range.
This is seen by comparison against the slope of the linear function 
$\psi_0(t)$ given by $\psi_0(0) = 1$ and $\psi_0(1) = 1/3$.
This implies a lower bound on the Lipschitz constant, 
namely $L_0 > (2/3)/u_1$.
The local condition for uniqueness in the present discussion given by
\eqref{eqn:1dun006a}, requires on the first element of the partition, 
$(u_1 - 0) < 2 k_\alpha/L$. Applying $k_\alpha=1/3$ and $L_0>(2/3)/u_1$, 
the condition for uniqueness on the first element requires
\begin{align}\label{eqn:1Dcounter003}
u_1 < \f{2 k_\alpha}{L_0} < \f{(2/3)u_1}{(2/3)} = u_1.
\end{align}
As such, the local condition \eqref{eqn:1dun006a} suffices to detect the possibility
of a nonunique solution.  
\begin{remark}
It is further noted that if \eqref{eqn:1Dcounter001} is modified so that
$k \le \phi(t) \le 1$, for any $0 < k < 1/3$, 
then $k_\alpha = k$, and $L_0 > (1-k)/u_1$, by the same reasoning as above.  
Then in place of
\eqref{eqn:1Dcounter003} we have $u_1 < 2k u_1/(1-k) < u_1$, violating the 
condition for uniqueness. More generally, for $0 < k < 1/2$,  the magnitude of the 
slope of the secant line from $(0,\phi(0))$ to $(t,\phi(t))$ must be at least
$|\phi(0) - \phi(t)|/t > (1-k)^2/(1-2k)$ for some 
$t \in (0,1)$ in order to simultaneously satisfy the 
first and last conditions of \eqref{eqn:1Dcounter001}. 
This is found by considering the area below the curve
$\psi(t) = \max\{k, 1-st\}$ from $t = 0$ to $t= 1$ with $s$ defined so the area $\int_0^1 \psi(t)=1/2$, 
namely, $s = -(1-k)^2/(1-2k)$.
Then $L_0 > |s|/u_1$ and  
$2k/L_0 < u_1 \cdot 2k(1-2k)/(1-k^2) < u_1$. 

\end{remark}
\section{Uniqueness in the two-dimensional case}
\label{sec:2D}
Let $\Omega \subset \R^2$ be an open bounded polyhedral domain, with boundary 
$\pa \Omega = \Gamma_D \cup \Gamma_N$, where $\Gamma_D$ has positive measure, and
$\Gamma_N = \pa \Omega \setminus \Gamma_D$.
Homogeneous Dirichlet boundary conditions will be imposed on $\Gamma_D$, and
Neumann conditions will be imposed on $\Gamma_N$.
We next establish the uniqueness of the piecewise linear finite element solution
to the discrete approximation corresponding to \eqref{eqn:PDE_strong},
in two dimensions, under the assumptions \eqref{A1:elliptic} and \eqref{A1:Lipschitz}. 
For simplicity of defining the finite element solution space, the discussion assumes a 
mixed boundary condition, with a homogeneous Dirichlet part.
However, the method of the proof trivially generalizes to allow nonhomogeneous 
Dirichlet data, or pure Dirichlet conditions, where $\Gamma_D = \pa \Omega$.

Let $\cT$ be a conforming triangulation of domain $\Omega$ by triangles that 
exactly captures the boundary $\pa \Omega$, and each of $\Gamma_D$ and $\Gamma_N$.  The mesh is assumed to satisfy the following acuteness condition, with the smallest angle to be bounded away from zero.
\begin{assumption}[Mesh regularity]
\label{A2:meshregularity}
There are numbers $ 0 < t_{min}< t_{max}$,  
for which the interior angles $\theta_i, ~i = 1,2,3$, 
of each $T \in \cT$ satisfy
\begin{align}\label{eqn:meshregularity}
t_{min} \le \theta_i \le t_{max} < \pi/ 2 , ~ i = 1, 2, 3.
\end{align}
\end{assumption}
Similar assumption on the angles of the mesh are made in
\cite{AnCh96a} (Section 4.).
The acuteness of each triangle disallows the orthogonality of the
gradients of the standard finite element basis functions. 
The smallest-angle assumption preserves the condition of the mesh, 
and necessarily for the discussion that follows, the minimum ratio of edges in a 
given triangle.  To that end, for each $T \in \cT$, denote the minimum ratio of 
sines of interior angles by
\begin{align}\label{eqn:gammadef}
\gamma_T & \coloneqq \min_{i, j \in \{1,2,3\}} \f{\sin(\theta_{i})}{\sin(\theta_{j})},
\text{ for } \theta_i, ~ i = 1,2,3, ~\text{ the angles of } T, 
\\ \label{eqn:cTdof}
     c_T & \coloneqq \min_{i = 1,2,3} \cos(\theta_i).
\end{align}
Under Assumption \ref{A2:meshregularity}, $\sin(t_{min}) \le \gamma_T \le 1$,
and $0<c_T <1$, for each $T \in \cT$.

Let $\cD$ be the collection of vertices or nodes of partition $\cT$, where
each $d \in \cD$ has coordinates $d = (x,y) \in \overline \Omega$.  
Define the discrete space $\cV_{0,D}$ as space of the piecewise linear
functions subordinate to partition $\cT$, that vanish on $\Gamma_D$. 
Without confusion, the following Section \ref{subsec:2Dsetup} 
represents $a = (x,y) \in \Omega$, as a coordinate representation, 
whereas the remainder of the text follows the convention $ x \in \Omega$ 
as a point in the domain with two associated physical coordinates.

\subsection{Basic setup in two dimensions}
\label{subsec:2Dsetup}
Some standard notations for the two dimensional problem are first reviewed, for 
the ease of presentation to the reader.
Let $\{a_1, a_2, a_3 \}$ $=\{(x_1,y_1), (x_2,y_2), (x_3,y_3) \}$ 
be a local counterclockwise numbering of the vertices of a simplex 
$T \in \cT$.  Let the corresponding edges $\{e_1, e_2, e_3 \}$, follow
a consistent local numbering, namely edge $e_i$ is opposite vertex 
$a_i, ~ i = 1, 2, 3$.
Let $\varphi_i$ the basis function on element $T \in \cT$ defined by its nodal
values at the vertices of $T$.
\begin{align*}
\varphi_i(a_j) 
= \left\{ \begin{array}{ll}
  1, & i = j, \\
  0, & i \ne j
  \end{array}\right., ~i,j = 1, 2, 3.
\end{align*}

The inner product between gradients of basis functions, and their respective 
integrals over elements $T \in \cT$, may be
computed by change of variables to a reference element $\widehat T$, in reference 
domain variables $(\widehat x, \widehat y)$.
Specifically, the coordinates of $\widehat T$ are given as 
$(\widehat x_1, \widehat y_1) = (0,0)$,
$(\widehat x_2, \widehat y_2)= (1, 0)$,
$(\widehat x_3, \widehat y_3)= (0,1)$.
The Jacobian of the transformation between reference and physical coordinates, 
and the corresponding Jacobian $J$ are given by
\begin{align}\label{eqn:jacobian}
J \left(\begin{array}{rr}\widehat x  \\ \widehat y \end{array} \right)= 
\left(\begin{array}{rr}x - x_1 \\ y - y_1\end{array} \right) , \quad
J = \left( \begin{array}{rrr}
    x_2 - x_1 & x_3 - x_1 \\
    y_2 - y_1 & y_3 - y_1
    \end{array}\right),
\end{align}
with $\det J = 2|T|$, where $|T|$ is the area of triangle $T$.
The reference element $\widehat T$ is equipped 
with the nodal basis functions $\widehat \varphi_i, ~i = 1, 2, 3.$, where
\begin{align*}
\widehat \varphi_1 = 1 - \widehat x - \widehat y, \quad 
\widehat \varphi_2 = \widehat x, \quad
\widehat \varphi_3 = \widehat y, 
\end{align*}
and gradients $\widehat \grad$ taken with respect to the reference domain variables
$\widehat x$ and $\widehat y$
\begin{align*}
\widehat \grad \widehat \varphi_1 = \left(\begin{array}{r}-1\\-1 \end{array} \right), 
\quad 
\widehat \grad \widehat \varphi_2 = \left(\begin{array}{r} 1\\0 \end{array} \right), 
\quad
\widehat \grad \widehat \varphi_3 = \left(\begin{array}{r} 0\\1 \end{array} \right). 
\end{align*}
Then $\grad \varphi_i = J^{-T} \widehat \grad \widehat \varphi_i$ 
and it is useful to note that
\begin{align}\label{eqn:grad_id}
\grad \varphi_i + \grad \varphi_j = -\grad \varphi_k,
\end{align} 
for any distinct assignment of $i,j$ and $k$ to the integers $\{1,2,3\}$.
Multiplication by the inverse of the Jacobian \eqref{eqn:jacobian}, allows 
the representation of $\grad \varphi_i^T \grad \varphi_i$ in terms of edge-length $|e_i|$.
The identity \eqref{eqn:grad_id} allows the expansion
$\grad \varphi_k^T \grad \varphi_k = 
 \grad(\varphi_i + \varphi_j)^T \grad(\varphi_i + \varphi_j)$,
and the representation of the inner product $\grad \varphi_i^T \grad \varphi_j$ 
in terms of the three edge-lengths of associated triangle $T$. 
\begin{align}\label{eqn:gradTgrad}
\grad \varphi_i^T \grad \varphi_j  
&= \f {1}{|\det J|^2}\left\{ \begin{array}{ll}
   |e_j|^2,                            & i = j, \\
  (|e_k|^2 - |e_i|^2 - |e_j|^2)/2 , & i \ne j.
  \end{array}\right.
\end{align}
The maximum interior angle of $\pi/2$ from Assumption \ref{A2:meshregularity},
together with \eqref{eqn:gradTgrad} assures the nonpositivity of any
$\grad \varphi_i^T \grad \varphi_j$, for $i \ne j$.
Further, the integral of each inner product over element $T$ in the physical domain 
then satisfies the following.
\begin{align} \label{eqn:int_gradTgrad}
\int_T \grad \varphi_i^T \grad \varphi_i =   \f 1 {4|T|} |e_i|^2, ~\an~
\int_{T}\nabla \phi_{i}^{T}\nabla \phi_{j} = -\f{|e_i||e_j|}{4|T|}\cos \theta_{k}. 
\end{align}
The minimum and maximum interior angle conditions in Assumption \ref{A2:meshregularity},
together with the law of sines, allows a bound on 
the minimum and maximum ratio of edges in a triangle in terms of $\gamma_T$,
defined in \eqref{eqn:gammadef}. For each $T \in \cT$
it holds that
\begin{align}\label{eqn:edge_rat}
\gamma_T|e_{j}| \le |e_i| \le \gamma_T^{-1} |e_j|, ~i,j = 1,2,3.
\end{align}
\subsection{Uniqueness in two dimensions}
\label{subsec:2Dunique}
We next prove two technical lemmas to support the comparison theorem in two
dimesions.  In contrast to the case of one dimension, both the cases of Dirichlet
and mixed Dirichlet/Neumann boundary conditions are handled together. This is due
in part to the additional complexities of the mesh partition in two dimensions, 
particularly that a given vertex may belong to many triangles $T \in \cT$, rather
than only two.  Because of this, we consider a test function supported on all triangles
in the domain over which the difference of solutions is positive at any
vertex. In Lemma \ref{lemma:T1}, we establish the necessary
estimates for Theorem \ref{thm:2dcompare} on any triangle where the difference 
of solutions, denoted $w$, is positive on exactly one vertex; 
in Lemma \ref{lemma:T2}, we establish the analogous estimate for the case of two 
positive vertices.

\begin{lemma}\label{lemma:T1}
Let Assumptions \ref{A1:elliptic} and \ref{A1:Lipschitz} hold, and 
let $u_1, u_2, w, v \in \cV_{0,D}$, subordinate to some fixed 
$\cT$ and  $\Gamma_D$. 
Suppose there is some triangle $T \in \cT$ for which 
$w(x)$ is positive at one vertex, and nonpositive at the other two. 
Assign the set of indices $\cN_T = \{i,j,k\}$ to take
the values $\{1,2,3\}$ so that $w(a_i) > 0$, and $0 \ge w(a_j) \ge w(a_k)$.  
Define the test function $v(x)$ by it's nodal values by $v(a_i) = 1$, 
and $v(a) = 0$ at all other mesh vertices $a \in \cD$.
Then the following inequalities hold.
\begin{align}\label{eqn:L1_b1}
\int_T \kappa(x,u_1) \grad w^T  \grad v 
   & \ge 
(w(a_i) - w(a_k)) \f{|e_i||e_k|}{4|T|} (k_\alpha \gamma_T c_T),
\end{align}
and
\begin{align}\label{eqn:L1_b2}
 \int_T (\kappa(x,u_2) - \kappa(x,u_1)) \grad u_2^T \grad v
 & \le (w(a_i) - w(a_k)) \f{|e_i||e_k|}{4|T|}\f{7L_0}{6} \left(1 + \gamma_T^{-1} \right)
   \nonumber \\ 
   & \times \max_{i',j'}|u_2(a_{i'}) - u_2(a_{j'})|. 
\end{align}
\end{lemma}
\begin{proof}
The test function $v$ is a basis function of $\cV_{0,D}$, namely
$v = \varphi_i$, with respect to the local numbering on $T$. 
The function $w$ may be expanded as a linear combination of 
basis functions $\varphi_i$, $\varphi_j$ and $\varphi_k$. 
An integral over $T$ on the left hand side
of \eqref{eqn:2dc_002a} satisfies
\begin{align}\label{eqn:2dc_003}
\int_T \kappa(x,u_1) \grad w^T  \grad v & = 
  \sum_{n \in \cN_T}
  w(a_n) \int_T \kappa(x,u_1) \grad \varphi_n^T \grad \varphi_i  
  \nonumber \\
& = w(a_i)\! \int_T\!\! \kappa(x,u_1)\grad \varphi_i^T  \grad \varphi_i  
  + w(a_j) \int_T\!\! \kappa(x,u_1)\grad (\varphi_j + \varphi_k)^T  \grad \varphi_i
  \nonumber \\ 
& + (w(a_k) - w(a_j)) \int_T \kappa(x,u_1)\grad \varphi_k^T \grad \varphi_i. 
\end{align}
Applying the identity \eqref{eqn:grad_id}, the 
first two terms of \eqref{eqn:2dc_003} may be combined.  The angle condition
\eqref{eqn:meshregularity} and the choice $w(a_k) \le w(a_j)$ assures the last term
is non-negative.  Together with the lower bound from \eqref{eqn:A1:elliptic} 
and integration by \eqref{eqn:int_gradTgrad} this yields
\begin{align}\label{eqn:2dc_004}
\int_T \kappa(x,u_1) \grad w^T  \grad v 
  & = \left( w(a_i) - w(a_j)\right)
  \int_T \kappa(x,u_1)\grad \varphi_i^T  \grad \varphi_i
  \nonumber \\ 
  & + \left( w(a_k) - w(a_j)\right)
  \int_T \kappa(x,u_1)\grad \varphi_k^T  \grad \varphi_i 
   \nonumber \\
   & \ge \f{k_\alpha c_T}{4|T|} \left\{(w(a_i) - w(a_j))|e_i|^2
    +   (w(a_j) - w(a_k))|e_i||e_k| \right\}. 
\end{align}
Controlling the minimum ratio of edges by $\gamma_T$, as in \eqref{eqn:edge_rat},
inequality \eqref{eqn:2dc_004} reduces to the first result, \eqref{eqn:L1_b1}.

To establish the second estimate, expand $u_2$ as a linear combination of
basis functions. Again applying $v = \varphi_i$ allows
\begin{align}\label{eqn:2dc_005}
\int_T  (\kappa(x,u_2) - \kappa(x,u_1)) \grad u_2^T \grad v
& = \sum_{n \in \cN_T }u_2(a_n) \int_{T}(\kappa(x,u_2) - \kappa(x,u_1)) 
  \grad \varphi_n^T \grad \varphi_i
  \nonumber \\
  & = u_2(a_i) \int_T (\kappa(x,u_2) - \kappa(x,u_1)) \grad \varphi_i^T 
     \grad \varphi_i
     \nonumber \\
  &  + u_2(a_j)
  \int_T (\kappa(x,u_2) - \kappa(x,u_1))\grad (\varphi_j + \varphi_k)^T \grad \varphi_i
  \nonumber \\
  &   + (u_2(a_k) - u_2(a_j))
  \int_T (\kappa(x,u_2) - \kappa(x,u_1))\grad \varphi_k^T  \grad \varphi_i.
\end{align}
Applying the identity
$\grad (\varphi_j + \varphi_k) = -\grad \varphi_i$, to \eqref{eqn:2dc_005}, obtain
\begin{align}\label{eqn:2dc_006}
 \int_T \!(\kappa(x,u_2) - \kappa(x,u_1)) \grad u_2^T \grad v
 & \le |u_2(a_i) - u_2(a_j)| 
    \left| \int_T \!(\kappa(x,u_2) - \kappa(x,u_1))\grad \varphi_i^T  
           \grad \varphi_i\right|
    \nonumber \\ 
  & + | u_2(a_j) - u_2(a_k)|
  \left| \int_T \!(\kappa(x,u_2) - \kappa(x,u_1)) \grad \varphi_k^T 
                                             \grad \varphi_i\right|.
\end{align}
Evaluating the constant inner product between basis functions by \eqref{eqn:gradTgrad}
and applying the Lipschitz condition \eqref{eqn:A1:Lipschitz}, 
inequality \eqref{eqn:2dc_006} reduces to
\begin{align}\label{eqn:2dc_006a}
 \int_T\!(\kappa(x,u_2) - \kappa(x,u_1)) \grad u_2^T \grad v
 & \le \f{L_0(|e_i|^2 + |e_i||e_k|)}{4|T|^2}\max_{i',j'}|u_2(a_{i'}) - u_2(a_{j'})| 
       \int_T |w|
 \nonumber \\
 & \le \f{L_0(|e_i||e_k|(1 + \gamma_T^{-1}))}{4|T|^2}
       \max_{i',j'}|u_2(a_{i'}) - u_2(a_{j'})| \int_T |w|,
\end{align}
where the last inequality follows from \eqref{eqn:edge_rat}.

To bound the integral of $|w|$, consider the following decomposition
making use of $\varphi_i + \varphi_j + \varphi_k = 1$, $w(a_i) \ge w(a_j)$ and 
$w(a_j) \ge w(a_k)$.
\begin{align*}
|w| &= |w(a_i) \varphi_i + w(a_j) \varphi_j + w(a_k) \varphi_k |
  \nonumber \\
    &= |(w(a_i)-w(a_j)) \varphi_i + (w(a_k) -w(a_j)) \varphi_k + 
         w(a_j)(\varphi_i + \varphi_j + \varphi_k) |
  \nonumber \\
    &\le (w(a_i) - w(a_j))\varphi_i + (w(a_j) - w(a_k))\varphi_k + (w(a_j) - w(a_k))
  \nonumber \\
    &= (w(a_i) - w(a_j))\varphi_i + (w(a_j) - w(a_k))(1+\varphi_k). 
\end{align*}
Applying $\int_T \varphi_i = \int_T \varphi_k = |T|/6$, we find
\begin{align}\label{eqn:2dc_008}
\int_T|w| \le (w(a_i) - w(a_j)) \f{|T|}{6} + (w(a_j) - w(a_k))\f{7|T|}{6}
          \le (w(a_i) - w(a_k)) \f{7|T|}{6}.
\end{align}
Putting together \eqref{eqn:2dc_006a} with \eqref{eqn:2dc_008} yields the 
desired estimate \eqref{eqn:L1_b2}, which is suitable for comparison with
the first result, \eqref{eqn:L1_b1}.
\end{proof}

We next consider the analogous estimates in the case where the function 
$w$ is positive at two vertices.  Notably, the results in this second case 
closely resemble the first.
There are however minor differences in calculations to obtain each estimate.
These two cases together are sufficient to prove the comparison theorem that follows.
\begin{lemma}\label{lemma:T2}
Let Assumptions \ref{A1:elliptic} and \ref{A1:Lipschitz} hold, and 
let $u_1, u_2, w, v \in \cV_{0,D}$, subordinate to some fixed 
$\cT$ and  $\Gamma_D$. 
Suppose there is some triangle $T \in \cT$ for which 
$w(x)$ is positive at two vertices, and nonpositive at the third. 
Assign the set of indices $\cN_T = \{i,j,k\}$ to take
the values $\{1,2,3\}$ so that $w(a_i) \ge w(a_j) > 0$, and $0 \ge w(a_k)$.  
Define the test function $v(x)$ by it's nodal values by $v(a_i) = v(a_j) = 1$, 
and $v(a) = 0$ at all other mesh vertices $a \in \cD$.
Then the following inequalities hold.
\begin{align}\label{eqn:L2_b1}
\int_T \kappa(x,u_1) \grad w^T  \grad v 
   &
\ge (w(a_i) - w(a_k)) \f{|e_i||e_k|}{4|T|} (k_\alpha \gamma_T c_T),
\end{align}
and
\begin{align}\label{eqn:L2_b2}
 \int_T (\kappa(x,u_2) - \kappa(x,u_1)) \grad u_2^T \grad v
 & \le (w(a_i) - w(a_k)) \f{|e_i||e_k|}{4|T|}\f{7L_0}{6}\left(1 +\gamma_T^{-1} \right)
   \nonumber \\ 
   & \times \max_{i',j'}|u_2(a_{i'}) - u_2(a_{j'})|.
\end{align}

\end{lemma}
\begin{proof}
In this case the test function $v$ is given by 
$\varphi_i + \varphi_j$.
The identity \eqref{eqn:grad_id} then implies $\grad v = - \grad \varphi_k$,
yielding
\begin{align}\label{eqn:2dc010}
\grad w^T \grad v  
  & = \sum_{n \in \cN_T} -w(a_n)\grad \varphi_n^T \grad \varphi_k
  \nonumber \\
  & = -(w(a_i) - w(a_j)) \grad \varphi_i^T \grad \varphi_k
      +(w(a_j) - w(a_k)) \grad \varphi_k^T \grad \varphi_k,
\end{align}
where the angle condition \eqref{eqn:meshregularity} implies the non-negativity 
of the first term, and the positivity of the second is clear from the local
ordering of the nodes. Applying the positivity of $\kappa(x,u_1)$ from \eqref{eqn:A1:elliptic},
the expansion \eqref{eqn:2dc010} and 
integration over $T$ using \eqref{eqn:int_gradTgrad} allows 
\begin{align}\label{eqn:2dc_011}
\int_{T} \kappa(x,u_1) \grad w^T \grad v 
  & \ge 
\f{k_\alpha c_T}{4|T|} \left\{(w(a_i) - w(a_j)) |e_i||e_k|
      +(w(a_j) - w(a_k)) |e_k|^2 \right\}.
\end{align}
As in the previous lemma, using \eqref{eqn:edge_rat} to bound the minimum ratio 
of edges in \eqref{eqn:2dc_011} in terms of  $\gamma_T$, we have the result
\eqref{eqn:L2_b1}.

For the second estimate, first expand $u_2$ as a linear combination of basis functions, 
and apply $\grad v = -\grad \varphi_k$.
\begin{align}\label{eqn:2dc_012}
-\grad u_2^T \grad v 
  & = \sum_{n \in \cN_T} u_2(a_n) \grad \varphi_n^T \grad \varphi_k
  \nonumber \\
  & = (u_2(a_i) - u_2(a_j)) \grad \varphi_i^T \grad \varphi_k
    +(u_2(a_k) - u_2(a_j)) \grad \varphi_k^T \grad \varphi_k.
\end{align}
Then, using the expansion of \eqref{eqn:2dc_012}, apply the Lipschitz condition
\eqref{eqn:A1:Lipschitz} to bound the integral over $T$.
\begin{align}\label{eqn:2dc_013}
\int_{T} (\kappa(x,u_2) - \kappa(x,u_1)) \grad u_2^T \grad v 
  & = (u_2(a_i) - u_2(a_j))\int(\kappa (x,u_1) - \kappa(x,u_2)) 
                           \grad \varphi_i^T \grad \varphi_k
  \nonumber \\
  & + (u_2(a_k) - u_2(a_j))\int(\kappa (x,u_1) - \kappa(x,u_2)) 
                           \grad \varphi_k^T \grad \varphi_k
  \nonumber \\
  & \le |u_2(a_i) - u_2(a_j)|\f{L_0|e_i||e_k|}{4|T|} \int_T|w|
  \nonumber \\  
  & +   |u_2(a_k) - u_2(a_j)|\f{L_0|e_k|^2}{4|T|} \int_T|w|
  \nonumber \\
 & \le \f{L_0(|e_k|^2 + |e_i||e_k|)}{4|T|^2}\max_{i',j'}|u_2(a_{i'}) - u_2(a_{j'})| 
       \int_T |w|.
\end{align}\label{eqn:2dc_013a}
The integral of $|w|$ over $T$ is also bounded by \eqref{eqn:2dc_008}.
Putting this together with
\eqref{eqn:2dc_013}, and applying $|e_k| \le \gamma_T^{-1}|e_i|$ from
\eqref{eqn:edge_rat}, allows 
the second part of the result, \eqref{eqn:L2_b2}, which is in suitable form
for comparison with the first part, \eqref{eqn:L2_b1}.
\end{proof}
With the technical lemmas in hand, we are now ready to prove a 
comparison theorem in two dimensions.  The corollary
to this theorem is uniqueness of the discrete solution under mixed 
Dirichlet/Neumann boundary conditions.  
\begin{theorem}[Comparison Theorem in 2D] 
\label{thm:2dcompare}
Let Assumptions \ref{A1:elliptic} and \ref{A1:Lipschitz} hold,
let $f_1 \an f_2 \in L_2(\Omega)$ satisfy $f_1 \le f_2, \ale x \in \Omega$;
and, let $\psi \in H^{1/2}(\Gamma_N)$.
Let $u_i \in \cV_{0,D}, ~i = 1,2$, denote the solution to the respective weak forms 
of \eqref{eqn:PDE_strong}, subject to the following boundary conditions. 
\begin{align}\label{eqn:2D_bc}
u = 0 \on \Gamma_D, \an \kappa(x,u) \grad u^T {\boldsymbol n} = \psi  \on \Gamma_N,
\end{align}
with $n$ the outward-facing normal on $\Gamma_N$.
In particular
\begin{align}\label{eqn:2Dweak_diri}
(\kappa(x,u_i) \grad u_i, \grad v) = (f_i,v)  + \int_{\Gamma_N} \psi v 
~\tforall v \in \cV_0, \quad i = 1,2.
\end{align} 
Then under the condition
\begin{align}\label{eqn:2D:var_cond}
\max_{T \in \cT } \left( \gamma_T^{-1}(1 + \gamma_T^{-1})
       \max_{i',j' \in \{1,2,3\}} |u_2(a_{i'}) - u_2(a_{j'})| \right) 
< \f{6 k_\alpha c_T}{7 L_0},
\end{align} 
it holds that $u_1 \le u_2$.
\end{theorem}
\begin{proof}
By \eqref{eqn:2Dweak_diri} and subtraction
\begin{align}\label{eqn:2dc_001}
(\kappa(u_1)\grad w, \grad v) = ((\kappa(u_2) - \kappa(u_1))\grad u_2, \grad v)
+ (f_1 - f_2, v),
\end{align} 
where $w = u_1 - u_2$. We prove the theorem by contradiction. Assume there is 
at least one node $d\in \cD$ on which $u_{1}(d) > u_{2}(d)$, i.e. $w(d) >0.$ 
Then the test function $v$ may be defined by its nodal values: 
one at nodes where $w$ is positive, and zero elsewhere. 
\begin{align}\label{eqn:vdef}
v(d) = \left\{ \begin{array}{ll}
1, & ~\text{ for } d \in \cD ~\text{ with }~ w(d) > 0, \\
0, & ~\text{ for } d \in \cD ~\text{ with }~ w(d) \le 0.
\end{array}\right.
\end{align}
Then $w$ is zero on $\Gamma_D$ so that $v \in \cV_{0,D}$. 
For each triangle $T \in \cT$, $v(a_i) = 0$, on either one,
two, or all three vertices $a_i, ~i = 1, 2, 3.$
Define the following subsets of partition $\cT$ with respect to this property.
\begin{align}\label{eqn:2dc_002}
\cT_1 &= \{ T \in \cT ~\rest~ v(a_i) = 1, ~\text{ on exactly one vertex }~ a_i \of T\}, \\
\cT_2 &= \{ T \in \cT ~\rest~ v(a_i) = 1, ~\text{ on exactly two vertices }~ a_i, a_j, 
          i \ne j, \of T\}.
\end{align}
Due to the fact that $\Gamma_{D} \neq \emptyset$, it is clear that $\cT_{1}\cup \cT_{2}$ is nonempty.  
As $\grad w = 0$, over $\cT \setminus \{\cT_1 \cup \cT_2\}$,
equation \eqref{eqn:2dc_001} can then be written in terms of elementwise integration
over partitions $\cT_1$ and $\cT_2$.
\begin{align}\label{eqn:2dc_002a}
\sum_{T \in (\cT_1 \cup \cT_2)} \int_T \kappa(u_1) \grad w^T \grad v = 
\sum_{T \in (\cT_1 \cup \cT_2)} \int_T (\kappa(u_2) - \kappa(u_1)) \grad u_2^T 
\grad v + (f_1 - f_2,v).
\end{align}
Bounds on the left and right hand sides of \eqref{eqn:2dc_001} have been
worked out  
by elementwise integration, first for $T \in \cT_1$, in Lemma \ref{lemma:T1},
then for $T \in \cT_2$ in Lemma \ref{lemma:T2}.
Putting together the results of these two lemmas,
\eqref{eqn:L1_b1}, \eqref{eqn:L1_b2}, \eqref{eqn:L2_b1} and \eqref{eqn:L2_b2}, 
into \eqref{eqn:2dc_002a}, we have
\begin{align}\label{eqn:2dc_015}
&
\sum_{T \in (\cT_1 \cup \cT_2)} (w(a_i) - w(a_k)) \f{|e_i||e_k|}{4|T|} 
     \left\{ k_\alpha \gamma_T c_T - \f{7 L_0}{6}\left(1 + \gamma_T^{-1} \right)
      \max_{i',j'}|u_2(a_{i'}) - u_2(a_{j'})|\right\}
\nonumber \\
\le& (f_1 - f_2,v) \le 0. 
\end{align}
As the terms to the left of the brackets on each line of \eqref{eqn:2dc_015} are
strictly positive, a contradiction is attained under the condition 
\eqref{eqn:2D:var_cond}.
This establishes the comparison result.
\end{proof}
It is emphasized that condition \eqref{eqn:2D:var_cond} is a local condition, 
and it differs from the one dimensional analogue \eqref{eqn:1dun003}, by 
a factor related to the mesh geometry and in particular to smallest
angle. As in one dimension, the two dimensional comparison theorem
leads to the main result.
\begin{corollary}[Uniqueness of the FE Solution in 2D]
\label{cor:2d_unique}
Let Assumptions \ref{A1:elliptic} and \ref{A1:Lipschitz} hold, 
let $f \in L_2(\Omega)$ and $\psi \in H^{1/2}(\Gamma_N)$.
Let $u \in \cV_{0,D}$, denote the solution to the weak form
of \eqref{eqn:PDE_strong}, subject to the following boundary conditions. 
\begin{align*}
u = 0 \on \Gamma_D, \an \kappa(x,u) \grad u^T {\boldsymbol n} = \psi  \on \Gamma_N.
\end{align*}
In particular
\begin{align}\label{eqn:2dweakpde}
(\kappa(x,u) \grad u, \grad v) = (f,v)  + \int_{\Gamma_N} \psi v 
~\tforall v \in \cV_{0,D}.
\end{align} 
Then under the condition 
\begin{align}\label{eqn:2D:var_condu}
\max_{T \in \cT } \left( \gamma_T^{-1}(1 + \gamma_T^{-1})
       \max_{i',j' \in \{1,2,3\}} |u(a_{i'}) - u(a_{j'})| \right) 
< \f{6 k_\alpha c_T}{7 L_0},
\end{align} 
({\em cf.,} \eqref{eqn:2D:var_cond}), the linear FE solution $u$
is the unique solution to \eqref{eqn:2dweakpde}.
\end{corollary}

\begin{proof}
Let $f_1 = f_2 = f$ in Theorem ~\ref{thm:2dcompare}. Then, let $u_1$ and $u_2$ be
two respective solutions to problem \eqref{eqn:2Dweak_diri}. Applying the results
of Theorem \ref{thm:2dcompare} yields both $u_1 \le u_2$ and $u_2 \le u_1$.  
Notably, the argument again does not
depend on either local or global meshsizes, only the variation in the discrete 
piecewise linear solution over each element. This can be controlled in an 
adaptive algorithm by monitoring
the nodal values over each element and bisecting triangles where the condition
for uniqueness is violated.
\end{proof}

\begin{remark}
A sufficient condition for uniqueness is then given by
\begin{align}\label{eqn:2D:alt_cond}
\max_{T \in \cT}\max_{i',j' \in \{1,2,3\}} |u(a_{i'}) - u(a_{j'})| 
< \f{6 k_\alpha}{7 L_0} \left( \f{s_{min}^2 c_{\min}}{1 + s_{min}} \right),
\end{align} 
where $s_{min} = \sin(t_{min})$, 
and $c_{min} = \min_{T\in\cT} c_T$, by applying the global bound $s_{min} \le \gamma_T \le 1$.  While this condition
appears easier to check if used in an adaptive algorithm, the condition 
\eqref{eqn:2D:var_condu} is preferred. In particular, the locally angle-dependent
version better characterizes
the condition for uniqueness based on both the maximum variation of solution $u$ over
an element $T$, together with the condition of that triangle.  
For instance, if the solution varies rapidly where the mesh is very regular, and
the solution varies slowly where the mesh is poorly conditioned in the sense of 
containing small angles, then \eqref{eqn:2D:alt_cond} gives an overly
pessimistic requirement.
\end{remark}

\subsection{Addressing the 2D counterexample of \cite{AnCh96a}}
\label{subsec:2Dcounter}
We refer the reader to \cite{AnCh96a} for the full details of a 2D Dirichlet problem
for which there are at least two discrete solutions on a given mesh.  We 
demonstrate here the uniqueness condition \eqref{eqn:2D:var_condu} is not satisfied
on the given mesh, detecting the possibility of nonuniqueness. 
The example and present analysis are similar to the one dimensional analogue. 

In this example, $\Omega$ is the interior of the equilateral triangle with sides 
of length 1. The mesh $\cT_h$ is a uniform partition,
for which each element is an equilateral triangle with sides length $h$. 
The diffusion coefficient $\kappa(x,z)$ is constructed on $\cT_h$ given a function 
$u \in \cV_{0,D} \subset H_0^1(\Omega)$ that satisfies $u(x) = 0$ for 
$x \in \Gamma_D = \pa \Omega$, and $u(x) > 0$  for $x \in \Omega$. 

The coefficient $\kappa(x,z)$ is constructed in terms of a $C^\infty(\Omega)$ function
$\phi(x)$ that satisfies 
\begin{align}\label{eqn:2Dcounter_001}
\phi = 1 \on \Gamma, 
\quad |\grad \phi| = 0 \on \Gamma, 
\quad \f 1 4 \le \phi \le 1, 
\quad \an ~\int_\Omega \phi = \f{\sqrt 3}{8}. 
\end{align}
Moreover, $\phi$ is assumed invariant to the isometries of $\Omega$ and defined 
locally in each element $T$ by change of variables to a reference element 
coinciding in shape with the global domain $\Omega$. Local coordinates in each element 
$T$ are described by 
$y_T: T \goto \Omega$, an invertible affine maps for each element $T \in \cT$ 
that assigns to each point 
$x \in T$ the point $y_T(x) \in \Omega$ with the same barycentric coordinates. 
It is outlined that
\begin{align}\label{eqn:2Dcounter_002}
\f{\pa}{\pa z} \kappa(x,z) = \f{\phi(y_T(x)) -1}{u(x)}, 
\quad u(x) \le z(x) \le 2u(x).
\end{align}
Now consider an element $T_1 \in \cT$ for which two vertices lie on $\pa \Omega$, 
and one vertex lies in the interior of $\Omega$. 
If the problem has any degrees of freedom, such an element must exist.
Denote the interior vertex as
$a_1$.  Then $u_1 \coloneqq u(a_1) > 0$. For clarity of presentation, assign 
the coordinate system $(t_1,t_2)$ to the reference domain so that $\Omega$ has
vertices  $\{(1/2, \sqrt 3/2), (0,0), (1,0)\}$.  
Let $y_T$ map vertex $a_1$ of $T_1$ to $(1/2, \sqrt 3/2)$.
Then $\phi(y_T(x)) = \phi((t_1, t_2))$ for each $x \in T$ and 
some $(t_1, t_2) \in \Omega$. 
Define $\breve \phi (t_2) \coloneqq \phi((1/2,t_2))$.
Then because $u(x)$ is an affine function in $T_1$, and $u(a_2) = u(a_3) = 0$, 
$u(x)$ on $T_1$ is dependent only on reference coordinate $t_2$, namely,
$u(y_T^{-1}(\cdot,t_2)) = (2/\sqrt 3) t_2 u_1$.
Apply this now to \eqref{eqn:2Dcounter_002} in 
triangle $T_1$, restricted to $x = y_T^{-1}((1/2,t_2))$, for $0 \le t_2 \le \sqrt 3/2$.

\begin{align}\label{eqn:2Dcounter_003}
\f{\phi((1/2,t_2)) - 1}{u(y_T^{-1}((1/2,t_2)))}
   = \f{\phi((1/2,t_2)) - \phi((1/2,0))}{ (2/\sqrt 3) t_2 u_1}
   = \left( \f{\sqrt 3}{2 u_1} \right) \f{\breve \phi(t_2) - \breve \phi(0)}{t_2}.
\end{align}

The definition of $\phi$ in \eqref{eqn:2Dcounter_001} with $k = 1/4$ 
makes sense for $0 < k < 1/2$, so we address $k$ in this range.
For any $0 <k < 1/2$, it suffices to consider the graph of $\breve \phi$ 
on the interval $0 \le t_2 \le \sqrt 3/2$.  The slope of the secant line given
in the right hand side term of \eqref{eqn:2Dcounter_003}, 
$(\breve \phi(t_2) - \breve \phi(0))/t_2,$ must somewhere
be greater in magnitude of the slope of the 
line connecting $\psi(0) = 1$ to $\psi(\sqrt 3/2) = k$, 
that is $2(1-k)/\sqrt 3$.  
Then the right side of \eqref{eqn:2Dcounter_003} must be
greater than $(1-k)/u_1$ for some $t_2$.
From \eqref{eqn:2Dcounter_002}, this shows the Lipschitz constant $L_0 > (1-k)/u_1$,
for $k_\alpha = k$.
For the equilateral mesh,
$\gamma_T = 1$ and  $c_T = \cos(\pi/3) = 1/2$ for all $T \in \cT$, 
and condition \eqref{eqn:2D:var_condu} then requires
\begin{align} 
u_1 < \f{3k}{14(1-k)} u_1 < u_1. 
\end{align}
This shows that
condition \eqref{eqn:2D:var_condu} detects the possibility of nonuniqueness.  

\section{Conclusion}
\label{sec:conclusion}
We demonstrated here uniqueness of the continuous piecewise linear finite element 
solution for a class of nonmonotone quasilinear elliptic problems under either 
Dirichlet or mixed Dirichlet/Neumann conditions. The main innovation 
is establishing these results without relying on the requirement of a 
globally fine mesh.  It has been observed in practice that this condition
while sufficient for the guarantee of well-posedness of the discrete 
problem, appears not to be necessary.  We established in this analysis
that it is sufficient for the variance of the solution over each 
element to be bounded by a multiple of the ratio of the ellipticity and Lipschitz 
constants for a given problem.  This result is important in the 
analysis of adaptive methods for this class of problems, as uniqueness of the
solution can now be assured based on local and computable criteria.


\bibliographystyle{abbrv}
\bibliography{refsTRN,maximum}



\end{document}